%% file: QRO_ArxivFinal.tex
\documentclass[letter]{amsart}
\usepackage[mathscr]{euscript}
\usepackage{graphicx}
\usepackage{amsfonts}
\usepackage{amscd}
\usepackage{amssymb}
\usepackage{amsthm}
\usepackage{mathtools}
\usepackage{xypic}
\xyoption{all}
\usepackage{verbatim}
\usepackage{hyperref}
\usepackage{url}
\usepackage{tikz}
\usepackage{tikz-cd}
\usepackage{calligra}
\usetikzlibrary{shapes}
\usetikzlibrary{decorations.markings}
\usepackage{multirow}
\usepackage[savepos]{zref}
\newtheorem{theorem}{Theorem}[section]
\newtheorem{corollary}[theorem]{Corollary}

\newtheorem{proposition}[theorem]{Proposition}
\numberwithin{theorem}{section}
\numberwithin{equation}{section}
\include{MathIncludes}

\DeclareMathOperator{\OR}{or}

\DeclareMathOperator{\uni}{uni}
\DeclareMathOperator{\Rank}{Rank}
\DeclareMathOperator{\cobdry}{d}

\begin{document}

\title{An induction theorem for groups acting on trees}%
\author{Martin H. Weissman}%
\address{Dept. of Mathematics, University of California, Santa Cruz, CA 95064}
\email{weissman@ucsc.edu}
\date{\today}
%\keywords{supercuspidal, sheaf, tree}
\maketitle

% ----------------------------------------------------------------
\begin{abstract}
If $G$ is a group acting on a tree $X$, and $\sheaf{S}$ is a $G$-equivariant sheaf of vector spaces on $X$, then its compactly-supported cohomology is a representation of $G$.  Under a finiteness hypothesis, we prove that if $H_c^0(X, \sheaf{S})$ is an irreducible representation of $G$, then $H_c^0(X, \sheaf{S})$ arises by induction from a vertex or edge stabilizing subgroup.

If $G$ is a reductive group over a nonarchimedean local field $F$, then Schneider and Stuhler realize every irreducible supercuspidal representation of $G(F)$ in the degree-zero cohomology of a $G(F)$-equivariant sheaf on its reduced Bruhat-Tits building $X$.  When the derived subgroup of $G$ has relative rank one, $X$ is a tree.  An immediate consequence is that every such irreducible supercuspidal representation arises by induction from a compact-mod-center open subgroup.
\end{abstract}

%\tableofcontents

\section{Introduction}
\label{intro}
According to a folklore conjecture, every irreducible supercuspidal representation of a reductive $p$-adic group arises by induction from a compact-mod-center open subgroup.  This is proven for $GL_n$ by Bushnell-Kutzko \cite{BK}, for many classical groups by Stevens \cite{Stevens}, and for tame supercuspidals by Ju-Lee Kim \cite{JLK} (exhaustion) and Jiu-Kang Yu \cite{JKY} (construction).  Outside of $GL_n$, these results require some assumptions on characteristic and residue characteristic.

Here we prove the conjecture for all groups of relative rank one -- those whose Bruhat-Tits building is a tree.  Our method is less constructive, but follows directly from results of Schneider-Stuhler \cite{SS} and the geometry of equivariant sheaves on trees.  No restrictions on residue characteristic (or characteristic!) are required, so the result seems new in many cases, e.g., for $SU_3$ in residue characteristic two and for quaternionic unitary groups.

There are nine classes of groups of relative rank one over a nonarchimedean local field $F$, if one uses the Tits index to organize them \cite[\S 4]{TitsCorvallis}.  These are conveniently tabulated and described in notes of Carbone \cite{Carbone}.  Their simply-connected forms are $SL_2(F)$ and $SL_2(D)$ (for $D$ a division algebra of any degree over $F$), the quasisplit unitary groups $SU_3^E$ and $SU_4^E$ ($E/F$ a separable quadratic field extension), and five types of quaternionic unitary groups $SU_2(D,s)$, $SU_3(D,s)$, $SU_3(D,h)$, $SU_4(D,h)$, $SU_5(D,h)$ which have absolute types $\mathrm{C}_2$, $\mathrm{C}_3$, $\mathrm{D}_3$, $\mathrm{D}_4$, and $\mathrm{D}_5$, respectively.   For these, $D$ denotes a quaternion division algebra over $F$, $h$ a nondegenerate Hermitian form, and $s$ a nondegenerate skew-hermitian form.  Previous results have addressed groups in three of these nine classes in a non-uniform manner, typically under restrictions on isogeny class, characteristic, and residue characteristic.

\subsection*{Acknowledgements}
The Simons Foundation Collaboration Grant \#426453 supported this work.   The author thanks Gordan Savin, Stephen DeBacker, Jeff Adler, Jeffrey Hakim, and Jessica Fintzen, for helpful discussions about representations of $p$-adic groups, buildings, and $K$-types.  The author also thanks Paul Broussous, who looked at an earlier version of the manuscript and  made helpful suggestions including a strengthening of the main theorem.

\section{Sheaves on Trees}
\label{sheavesontrees}
Let $X$ be a tree with vertex set $V$ and edge set $E$.  If $v \in V$ and $e \in E$, then we write $v < e$ to mean that $v$ is an endpoint of $e$.  Fix a field $k$.  A {\em sheaf} on $X$ will mean a cellular sheaf of $k$-vector spaces on $X$.  Such a sheaf consists of $k$-vector spaces $\sheaf{S}_v$ for every vertex $v \in V$, and $\sheaf{S}_e$ for every edge $e \in E$, and linear maps
$$\gamma_{v,e} \From \sheaf{S}_v \To \sheaf{S}_e$$
for all $v < e$.  The maps $\gamma_{v,e}$ are called {\em restriction maps} and the spaces $\sheaf{S}_v$, $\sheaf{S}_e$ are called the {\em stalks} of $\sheaf{S}$.  We write $(\sheaf{S}, \gamma)$ or sometimes just $\sheaf{S}$ for such a sheaf.

Let $G$ be a group acting on $X$.  A $G$-equivariant structure on a sheaf $(\sheaf{S}, \gamma)$ consists of linear maps
$$\eta_{g,v} \From \sheaf{S}_v \To \sheaf{S}_{gv}, \quad \eta_{g,e} \From \sheaf{S}_e \To \sheaf{S}_{ge}$$
for all $g \in G$, $v \in V$, $e \in E$, satisfying the following axioms.
\begin{itemize}
\item
For all $v \in V$, $e \in E$, the linear maps $\eta_{1,v}$ and $\eta_{1,e}$ are the identity.
\item
For all $g,h \in G$, $v \in V$, and $e \in E$, $\eta_{g,hv} \circ \eta_{h,v} = \eta_{gh,v}$ and $\eta_{g,he} \circ \eta_{h,e} = \eta_{gh,e}$.
\item
For all $g \in G$, and $v < e$, we have $\gamma_{gv,ge} \circ \eta_{g,v} = \eta_{g,e} \circ \gamma_{v,e}$.
\end{itemize}
A $G$-equivariant sheaf on $X$ will mean a sheaf $(\sheaf{S}, \gamma)$ endowed with a $G$-equivariant structure.

\subsection{Cohomology}

For convenience, fix an orientation on every edge $e \in E$.  This orders the endpoints of every edge $e$, and we write $x_e, y_e$ for the first and second endpoint, respectively.  If $v < e$, then write $\OR(v,e) = 1$ if $v = x_e$ and $\OR(v,e) = -1$ if $v = y_e$.

Fix a sheaf $(\sheaf{S}, \gamma)$ on $X$ in what follows.  If $v \in V$, and $s \in \sheaf{S}_v$, define
$$\cobdry s = \sum_{e > v} \OR(v,e) \cdot \gamma_{v,e}(s) \in \bigoplus_{e > v} \sheaf{S}_e.$$
The compactly-supported cohomology of $\sheaf{S}$ is then computed by the complex
$$0 \To \bigoplus_{v \in V} \sheaf{S}_v \xrightarrow{\cobdry} \bigoplus_{e \in E} \sheaf{S}_e \rightarrow 0.$$
With reference to the complex above,
$$H_c^0(X, \sheaf{S}) = \Ker \cobdry, \quad H_c^1(X, \sheaf{S}) = \Cok \cobdry.$$

If $\sheaf{S}$ is a $G$-equivariant sheaf, then the complex above and its cohomology inherit actions of $G$.  In particular, $H_c^i(X, \sheaf{S})$ is a representation of $G$ on a $k$-vector space for $i = 0,1$.

\subsection{The elliptic subsheaf}

Fix a  $G$-equivariant sheaf $(\sheaf{S}, \gamma, \eta)$ on $X$ in what follows.  If $v \in V$ and $s \in \sheaf{S}_v$, we say that $s$ is {\em elliptic} if $\cobdry s = 0$, i.e., if
$$\gamma_{v,e}(s) = 0 \text{ for all } e > v.$$
The elliptic elements of $\sheaf{S}_v$ form a subspace $\sheaf{S}_v^{\ellip} \subset \sheaf{S}_v$.  If $e$ is an edge, define $\sheaf{S}_e^{\ellip} = 0$.  This defines a $G$-equivariant subsheaf of $\sheaf{S}$, which we call the {\em elliptic subsheaf}
$$\sheaf{S}^{\ellip} \subset \sheaf{S}.$$
By construction, we have
\begin{equation}
H_c^0(X, \sheaf{S}^{\ellip}) = \bigoplus_{v \in V} \sheaf{S}_v^{\ellip}, \text{ and } H_c^1(X, \sheaf{S}^{\ellip}) = 0.
\end{equation}

For $x \in V$, let $G_x$ denote its stabilizer and $G \cdot x$ its orbit.  Then $\sheaf{S}_x$ is naturally (via $\eta$) a representation of $G_x$, and $\sheaf{S}_x^{\ellip}$ is a $G_x$-subrepresentation.  Algebraic induction gives a natural identification of $G$-representations,
$$\bigoplus_{v \in G \cdot x} \sheaf{S}_v^{\ellip} \ident \cInd_{G_x}^G \sheaf{S}_x^{\ellip}.$$
Above and in what follows, if $K$ is a subgroup of $G$ and $(\sigma, S)$ is a representation of $K$, $\cInd_K^G(S)$ denotes ``algebraic induction,'' i.e., the space of functions $f \From G \To S$ supported on a finite number of left $K$-cosets, satisfying $f(gk) = \sigma(k)^{-1} f(g)$ for all $g \in G$, $k \in K$.

It follows that $H_c^0(X, \sheaf{S}^{\ellip})$ is a direct sum of such induced representations.
\begin{equation}
\label{EllInd}
H_c^0(X, \sheaf{S}^{\ellip}) \ident \bigoplus_{G \cdot x \in G \backslash V} \cInd_{G_x}^G \sheaf{S}_x^{\ellip}.
\end{equation}

\subsection{The unifacial subsheaf}

Suppose now that $\sheaf{S}^{\ellip} = 0$.  If $v \in V$ and $s \in \sheaf{S}_v$, we say that $s$ is {\em unifacial} if there exists a {\em unique} edge $e > v$ such that $\gamma_{v,e}(s) \neq 0$.  Define
$$\sheaf{S}_{v,e}^{\uni} = \Span_k \{ s \in \sheaf{S}_v : s \text{ is unifacial and } \gamma_{v,e}(s) \neq 0 \}.$$ 
Since we assume $\sheaf{S}^{\ellip} = 0$, we find that
$$\sheaf{S}_{v,e}^{\uni} = \{ s \in \sheaf{S}_v : s \text{ is unifacial and } \gamma_{v,e}(s) \neq 0 \} \sqcup \{ 0 \}.$$ 
The spaces $\sheaf{S}_{v,e}^{\uni}$, for various $e$, are linearly independent in $\sheaf{S}_v$.  Hence putting these spaces together for various edges, we define
\begin{equation}
\label{diruni}
\sheaf{S}_v^{\uni} = \Span_k \{ s \in \sheaf{S}_v : s \text{ is unifacial} \} = \bigoplus_{e > v} \sheaf{S}_{v,e}^{\uni}.
\end{equation}

If $e$ is an edge with endpoints $x,y$, define
$$\sheaf{S}_e^{\uni} = \gamma_{x,e}(\sheaf{S}_{x,e}^{\uni}) + \gamma_{y,e}(\sheaf{S}_{y,e}^{\uni}) \subset \sheaf{S}_e.$$
The spaces $\sheaf{S}_v^{\uni}$ and $\sheaf{S}_e^{\uni}$ define a $G$-equivariant subsheaf $\sheaf{S}^{\uni} \subset \sheaf{S}$, whose cohomology can be described explicitly.

\begin{proposition}
\label{UniVanish}
$H_c^1(X, \sheaf{S}^{\uni}) = 0$.
\end{proposition}
\begin{proof}
We demonstrate that $\cobdry$ is surjective as follows.  Suppose $e \in E$ and $s \in \sheaf{S}_e^{\uni}$.  Write $x = x_e$ and $y = y_e$.  Then there exists $a \in \sheaf{S}_{x,e}^{\uni}$ and $b \in \sheaf{S}_{y,e}^{\uni}$ such that 
$$s = \gamma_{x,e}(a) + \gamma_{y,e}(b).$$
Since $\gamma_{x,e'}(a) = 0$ and $\gamma_{y,e'}(b) = 0$ for all $e' \neq e$, we find that $s = \cobdry(a - b)$.  Hence $\cobdry$ is surjective and $H_c^1(X, \sheaf{S}^{\uni})$ vanishes.
\end{proof}

\begin{proposition}
$H_c^0(X, \sheaf{S}^{\uni}) = \bigoplus_{e \in E} \left( \gamma_{x_e,e}(\sheaf{S}_{x_e,e}^{\uni}) \cap \gamma_{y_e,e}(\sheaf{S}_{y_e,e}^{\uni}) \right)$.
\end{proposition}
\begin{proof}
We define auxiliary sheaves $(\sheaf{R}, \rho)$ and $(\sheaf{T}, \tau)$ as follows.  For every vertex $v$, define $\sheaf{R}_v = \sheaf{S}_v^{\uni}$ and $\sheaf{T}_v = 0$.  For every edge $e$, with endpoints $x,y$, define 
$$\sheaf{R}_e = \gamma_{x,e}(\sheaf{S}_{x,e}^{\uni}) \oplus \gamma_{y,e}(\sheaf{S}_{y,e}^{\uni}), \quad \sheaf{T}_e = \gamma_{x,e}(\sheaf{S}_{x,e}^{\uni}) \cap \gamma_{y,e}(\sheaf{S}_{y,e}^{\uni}).$$
For $v < e$, define $\tau_{v,e} \From \sheaf{T}_v \To \sheaf{T}_e$ to be the zero map.  Define $\rho_{v,e} \From \sheaf{R}_v \To \sheaf{R}_e$ by $\rho_{v,e}(s) = \gamma_{v,e}(s)$, where the latter is viewed in the summand $\gamma_{v,e}(\sheaf{S}_{v,e}^{\uni})$ of $\sheaf{R}_e$.

There is a natural short exact sequence of sheaves on $X$,
$$\sheaf{T} \Into \sheaf{R} \Onto \sheaf{S}^{\uni}.$$
At vertices, the maps are obvious; for edges, the map $\sheaf{T}_e \To \sheaf{R}_e$ sends an element
$$t \in \gamma_{x,e}(\sheaf{S}_{x,e}^{\uni}) \cap \gamma_{y,e}(\sheaf{S}_{y,e}^{\uni})$$
to the ordered pair $(t, -t)$.  The map from $\sheaf{R}_e$ to $\sheaf{S}_e^{\uni}$ is addition.  

Essentially by construction, $H_c^0(X, \sheaf{R}) = H_c^1(X, \sheaf{R}) = 0$.  Indeed, we can decompose $\sheaf{R}$ as a product of sheaves,
$$\sheaf{R} = \prod_{v \in V} \sheaf{R}^{(v)},$$
where $\sheaf{R}^{(v)}$ is supported on the star-neighborhood of $v$:  $\sheaf{R}_v^{(v)} := \sheaf{R}_v = \sheaf{S}_v^{\uni}$ and $\sheaf{R}_e^{(v)} = \gamma_{v,e}(\sheaf{S}_{v,e}^{\uni})$ for all $e > v$.  Since $\gamma_{v,e} \From \sheaf{S}_{v,e}^{\uni} \To \gamma_{v,e}(\sheaf{S}_{v,e}^{\uni})$ is an isomorphism for all $v < e$ (recall $\sheaf{S}^{\ellip} = 0$), we find that
$$\cobdry \From \sheaf{R}_v^{(v)} \longrightarrow \bigoplus_{e > v} \sheaf{R}_e^{(v)}$$
is an isomorphism (see \eqref{diruni}).  Hence $H_c^i(X, \sheaf{R}^{(v)}) = 0$ for $i = 0,1$.  The compactly supported cohomology of $\sheaf{R}$ is the direct sum of these, which vanishes.

From the short exact sequence of sheaves, $\sheaf{T} \Into \sheaf{R} \Onto \sheaf{S}^{\uni}$, the long exact sequence in cohomology gives an identification,
$$H_c^0(X, \sheaf{S}^{\uni}) \ident H_c^1(X, \sheaf{T}).$$
Since $\sheaf{T}_v = 0$ for all vertices $v$, we find that 
$$H_c^1(X, \sheaf{T}) = \bigoplus_{e \in E} \sheaf{T}_e = \bigoplus_{e \in E} \left( \gamma_{x_e,e}(\sheaf{S}_{x_e,e}^{\uni}) \cap \gamma_{y_e,e}(\sheaf{S}_{y_e,e}^{\uni}) \right).$$
\end{proof}

A $G$-equivariant structure on $\sheaf{S}$ transports to $G$-equivariant structures on $\sheaf{T}$ and $\sheaf{R}$.  It follows that, for any edge $e \in E$, the space $\sheaf{T}_e$ is naturally a representation of the stabilizer $G_e$.  Therefore, we find an identification of representations of $G$,
\begin{equation}
\label{UniInd}
H_c^0(X, \sheaf{S}^{\uni}) \ident H_c^1(X, \sheaf{T}) \ident \bigoplus_{G \cdot e \in G \backslash E} \cInd_{G_e}^G \sheaf{T}_e.
\end{equation}

\subsection{Multifacial sheaves}

Now, suppose that $(\sheaf{S}, \gamma)$ is a sheaf on $X$ and $\sheaf{S}^{\ellip} = 0$ and $\sheaf{S}^{\uni} = 0$.  Thus, for every $v \in V$ and every nonzero $s \in \sheaf{S}_v$, there exist {\em at least two} edges $e,f \in E$ such that $v<e$ and $v<f$ and 
$$\gamma_{v,e}(s) \neq 0 \text{ and } \gamma_{v,f}(s) \neq 0.$$
We call such a sheaf {\em multifacial}.
\begin{proposition}
\label{MultiVanish}
If $\sheaf{S}$ is a multifacial sheaf, then $H_c^0(X, \sheaf{S}) = 0$.
\end{proposition}
\begin{proof}
Suppose that $s = (s_v : v \in V) \in H_c^0(X, \sheaf{S})$.  Assume that $s \neq 0$.  Then there exists a nonempty finite set $W \subset V$ of vertices such that $s_v \neq 0$ if and only if $v \in W$.  Let $\Omega$ be the convex hull of $W$ in the tree $X$.  In other words, $\Omega$ is the smallest connected subgraph of $X$ containing every vertex from $W$.  In particular, $\Omega$ is a {\em finite tree}.  Moreover, if $\ell$ is a leaf of $\Omega$, then $\ell \in W$; otherwise one could prune the leaf while maintaining connectedness and containment of $W$.  

Let $\ell$ be a leaf of $\Omega$, so there is at most one edge of $\Omega$ having $\ell$ as an endpoint.  Since $\ell \in W$, we have $s_\ell \neq 0$.  Since $\sheaf{S}$ is multifacial, there exists an edge $e$ such that $\ell < e$, $e$ does not belong to $\Omega$, and $\gamma_{\ell,e}(s_\ell) \neq 0$.  Let $v$ be the other endpoint of $e$.  Since $\ell$ was a leaf of $\Omega$, $v \not \in \Omega$.

\begin{figure}[htbp!]
\begin{center}
\begin{tikzpicture}
\coordinate (ell) at (0,0);
\coordinate (v) at (-2,0);
\coordinate (n1) at (1,1);
\coordinate (n2) at (3,1);
\coordinate (n3) at (4.5,1);
\coordinate (n4) at (5,0);
\coordinate (n5) at (3.5, -1);
\coordinate (n6) at (2,-1);
\coordinate (n7) at (1.7,0);
\coordinate (n8) at (3.5,0);
\coordinate (p1) at (1.5,2);
\coordinate (p2) at (4, 1.5);
\coordinate (p3) at (2.5, 1.5);
\coordinate (p4) at (1, -1.5);
\filldraw[fill=black!10, draw=black!50] (0,-1) rectangle (5,1);
\draw (0.5, 0.5) node {\Large $\Omega$};
\draw (ell) -- (n7);
\draw (n7) -- (n1);
\draw (n7) -- (n2);
\draw (n7) -- (n6);
\draw (n7) -- (n8);
\draw (n8) -- (n3);
\draw (n8) -- (n4);
\draw (n8) -- (n5);
\draw (n1) -- (p1);
\draw (n2) -- (p2);
\draw (n2) -- (p3);
\draw (n6) -- (p4);
\draw (ell) to node[above] {$e$} (v);
\foreach \i in {1,...,8}
{
\filldraw[fill=gray, draw=black] (n\i) circle (0.05);
}
\foreach \i in {1,...,4}
{
\filldraw[fill=gray, draw=black] (p\i) circle (0.05);
}
\filldraw[fill=gray, draw=black] (ell) circle (0.05) node[above left] {$\ell$};
\filldraw[fill=gray, draw=black] (v) circle (0.05) node[above left] {$v$};; 

\end{tikzpicture}
\end{center}
\caption{$\ell$ is a leaf of the finite tree $\Omega$ (the subgraph contained in the shaded box), and $e$ is an edge that protrudes outside of $\Omega$.}
\label{fig:1}
\end{figure}
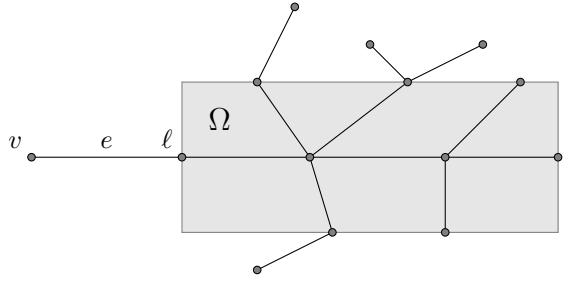

Since $\cobdry s = 0$, we have $(\cobdry s)_e = 0$.  Equivalently,
$$\gamma_{v,e}(s_v) - \gamma_{\ell,e}(s_\ell) = 0.$$
Since $\gamma_{\ell,e}(s_\ell) \neq 0$, this implies $\gamma_{v,e}(s_v) \neq 0$, which implies $s_v \neq 0$.  But this contradicts the fact that $v \not \in \Omega$.  

This contradiction proves that $H_c^0(X, \sheaf{S}) = 0$.
\end{proof}
  
\subsection{The Induction Theorem}

Suppose that $\sheaf{S}$ is a $G$-equivariant sheaf on $X$.  We define its {\em 0-rank} to be the cardinal number
$$\Rank^0(\sheaf{S}) = \sum_{G \cdot v \in G \backslash V} \dim(\sheaf{S}_v).$$
For example, if $G \backslash V$ is finite and every stalk of $\sheaf{S}$ is finite-dimensional, then $\Rank^0(\sheaf{S})$ will be finite.
\begin{theorem}
Assume that $\Rank^0(\sheaf{S})$ is finite.  If $H_c^0(X, \sheaf{S}) = 0$ or $H_c^0(X, \sheaf{S})$ is an irreducible representation of $G$, then $H_c^0(X, \sheaf{S})$ is isomorphic to a representation induced from the stabilizer of a vertex or edge.
\end{theorem}  
\begin{proof}
We proceed by induction on $\Rank^0(\sheaf{S})$.  If $H_c^0(X, \sheaf{S}) = 0$, then $H_c^0(X, \sheaf{S})$ is induced from the zero representation via any subgroup, and the result is trivial.  This takes care of the $\Rank^0(\sheaf{S}) = 0$ base case, in particular.

So assume that $\Rank^0(\sheaf{S}) > 0$ and the result has been proven for lower 0-rank.  If $H_c^0(X, \sheaf{S}) = 0$, then we are done.  Otherwise, by Proposition \ref{MultiVanish}, we find that $\sheaf{S}^{\ellip} \neq 0$ or ($\sheaf{S}^{\ellip} = 0$ and $\sheaf{S}^{\uni} \neq 0$).  We consider these two cases below.

If $\sheaf{S}^{\ellip} \neq 0$, then $H_c^0(X, \sheaf{S}^{\ellip})$ is a nonzero subrepresentation of $H_c^0(X, \sheaf{S})$.  By irreducibility, we find that 
$$H_c^0(X, \sheaf{S}) \ident H_c^0(X, \sheaf{S}^{\ellip}).$$
By \eqref{EllInd}, this is a direct sum of representations induced from stabilizers of vertices.  By irreducibility again, only one $G$-orbit of vertices can support $\sheaf{S}^{\ellip}$ and
$$H_c^0(X, \sheaf{S}) \ident \cInd_{G_x}^G \sheaf{S}_x^{\ellip}$$
for some vertex $x \in V$.  Thus if $\sheaf{S}^{\ellip} \neq 0$, the result holds.

Next, suppose that $\sheaf{S}^{\ellip} = 0$ and $\sheaf{S}^{\uni} \neq 0$.  Consider the short exact sequence,
$$\sheaf{S}^{\uni} \Into \sheaf{S} \Onto \sheaf{S} / \sheaf{S}^{\uni}.$$
Since $\sheaf{S}^{\uni} \neq 0$, $\sheaf{S}_v^{\uni} \neq 0$ for some vertex $v$, and so $\Rank^0(\sheaf{S} / \sheaf{S}^{\uni}) < \Rank^0(\sheaf{S})$.

By Proposition \ref{UniVanish}, the long exact sequence in cohomology yields
$$0 \To H_c^0(X, \sheaf{S}^{\uni}) \To H_c^0(X, \sheaf{S} ) \To H_c^0(X, \sheaf{S} / \sheaf{S}^{\uni} ) \To 0.$$
This is a short exact sequence of $G$-representations, so irreducibility of the middle term yields
$$H_c^0(X, \sheaf{S}) \ident H_c^0(X, \sheaf{S}^{\uni}) \text{ or } H_c^0(X, \sheaf{S}) \ident H_c^0(X, \sheaf{S} / \sheaf{S}^{\uni} ).$$
In the first case, \eqref{UniInd} and irreducibility yields
$$H_c^0(X, \sheaf{S}) \ident \cInd_{G_e}^G \sheaf{T}_e$$
for some edge $e \in E$ and the result holds.  

In the second case, $H_c^0(X, \sheaf{S}) \ident H_c^0(X, \sheaf{S} / \sheaf{S}^{\uni} )$.  But $\sheaf{S} / \sheaf{S}^{\uni}$ has lower 0-rank.  Hence the second case follows from the theorem for sheaves of lower 0-rank. 
\end{proof}

\begin{corollary}
Let $\alg{G}$ be a reductive group over a nonarchimedean local field $F$, whose derived subgroup has relative rank one.  Let $G = \alg{G}(F)$.  Then every irreducible supercuspidal representation of $G$ on a complex vector space is isomorphic to $\cInd_K^G \sigma$ for some compact-mod-center open subgroup $K \subset G$ and some irreducible representation $\sigma$ of $K$.
\end{corollary}
\begin{proof}
Let $X$ be the Bruhat-Tits tree of $\alg{G}$ over $F$.  Then $G$ acts on $X$ with finitely many orbits on vertices and edges.  The stabilizer of any vertex or any edge is a compact-mod-center open subgroup of $G$; every compact-mod-center open subgroup is contained in such a stabilizer.  

Let $(\pi, V)$ be an irreducible supercuspidal representation of $G$.  In \cite[\S IV.1 and Theorem IV.4.17]{SS}, Schneider and Stuhler construct a $G$-equivariant sheaf $\sheaf{S}$ on $X$, with finite-dimensional stalks, such that $H_c^0(X, \sheaf{S}) \ident V$.  The result now follows from the previous theorem.
\end{proof} 

Note that Vigneras \cite[Theorem 4.6]{Vig} has proven that the main results of Schneider and Stuhler adapt to representations of $G$ on $R$-vector spaces, when $R$ is a field whose characteristic is coprime to the residue characteristic of $F$.  Hence the previous result holds for cuspidal (see \cite[\S 3.7]{Vig} for the definition) $R$-representations as well as for supercuspidal complex representations.

% -------------------------------------------0---------------------
\bibliographystyle{amsalpha}
\bibliography{QRO}
% ----------------------------------------------------------------

\end{document}

%% file: MathIncludes.tex
\usepackage[mathscr]{euscript}
\usepackage[T2A,T1]{fontenc}

  {\begin{description}%
    \setlength{\itemsep}{2.5pt}%
    \setlength{\parskip}{5pt}}%
  {\end{description}}

\DeclareMathOperator{\ellip}{ell}

\DeclareMathOperator{\cInd}{cInd}

\DeclareMathOperator{\Span}{Span}

\DeclareMathOperator{\Ker}{Ker}
\DeclareMathOperator{\Cok}{Cok}

\newcommand{\From}{\colon}
\newcommand{\inar}{\ar@{^{(}->}}
\newcommand{\onar}{\ar@{->>}}

\usepackage{accents}
\newlength{\dtildeheight}

\makeatletter
\newcommand{\raisemath}[1]{\mathpalette{\raisem@th{#1}}}
\newcommand{\raisem@th}[3]{\raisebox{#1}{$#2#3$}}
\makeatother

\newcommand{\alg}[1]{\boldsymbol{\mathrm{#1}}}

\newcommand{\sheaf}[1]{{\mathscr{#1}}}

\newcommand{\ident}{\equiv}

\newcommand{\Into}{\hookrightarrow}
\newcommand{\Onto}{\twoheadrightarrow}
\newcommand{\To}{\rightarrow}

\makeatletter
\newcommand\@biprod[1]{%
  \vcenter{\hbox{\ooalign{$#1\prod$\cr$#1\coprod$\cr}}}}
\newcommand\biprod{\mathop{\mathpalette\@biprod\relax}\displaylimits}
\makeatother

\DeclareMathAlphabet{\mathcalligra}{T1}{calligra}{m}{n}